\documentclass[12pt]{amsart}
\usepackage{amsmath,latexsym,amscd,amsbsy,amssymb,amsfonts,amsthm,fleqn,leqno,
euscript, graphicx, texdraw, pb-diagram}
\usepackage{mathrsfs}
\usepackage[matrix,arrow,curve]{xy}

\newtheorem{thm}{Theorem}[section]
\newtheorem{pro}[thm]{Proposition}
\newtheorem{lem}[thm]{Lemma}

\newtheorem{cor}[thm]{Corollary}

\def\leukfrac#1/#2{\leavevmode
               \kern.1em
                \raise.9ex\hbox{\the\scriptfont0 ${}_#1$}
                \hskip -1pt\kern-.1em
                /\kern-.15em\lower.10ex\hbox{\the\scriptfont0 ${}_#2$}}

\theoremstyle{definition}
\newtheorem{re}[thm]{Remark}

\theoremstyle{remark}
\newtheorem{claim}{Claim}

\makeatletter
\def\Int{\mathop{\operator@font Int}\nolimits}
\makeatother

\begin{document}


\title[Homology manifolds and homogeneous compacta]
{Homology manifolds and homogeneous compacta}

\author{V. Valov}
\address{Department of Computer Science and Mathematics,
Nipissing University, 100 College Drive, P.O. Box 5002, North Bay,
ON, P1B 8L7, Canada} \email{veskov@nipissingu.ca}

\date{\today}
\thanks{The author was partially supported by NSERC
Grant 2025-07173.}


\begin{abstract}
A non-trivial separable metric space $X$ is called an almost homology $n$-manifold if the homology groups $H_k(X,X\backslash\{x\},\mathbb Z)$ are trivial for all $x\in X$ and all $k=0,1,..,n-1$. We provide a necessary and sufficient condition locally compact homogeneous $ANR$-spaces or strongly locally homogeneous $ANR$-spaces to be almost homology $n$-manifolds.  
\end{abstract}
\maketitle\markboth{}{Homology manifolds}





\section{Introduction}

All spaces are supposed to be locally connected, separable completely metrizable and all maps are continuous. Recall that a space $X$ is homogeneous if for every $x,y\in X$ there exists a homeomorphism $h$ on $X$ with $h(x)=y$. A space $X$ is strongly locally homogeneous if every point $x\in X$ has a base of neighborhoods $U$ such that for every $y,z\in U$ there exists a homeomorphism
$h$ on $X$ satisfying the following condition: $h(y)=z$ and $h$ is the identity on $X\backslash U$. Any region of a strongly locally homogeneous space is homogeneous, strongly locally homogeneous and locally connected \cite{ba}. Basic examples of strongly locally homogeneous spaces of dimension $n$ are the Euclidean space 
$\mathbb R^n$ and the Menger universal $n$-dimensional continuum $\mu_n$ \cite{be}.

Unless stated otherwise the homology groups are singular with integer coefficients $\mathbb Z$. Recall that a homology $n$-manifold is a locally compact $n$-dimensional space $X$ such that $H_k(X,X\backslash\{x\})=0$ for all $x\in X$ and $k<n$ and  $H_n(X,X\backslash\{x\})=\mathbb Z$. Homology manifolds are important objects in geometric topology and arise in many investigations of structural properties of manifolds. The interest to the relation between homology manifolds and homogeneous spaces is based on the following two conjectures.
\begin{itemize}
\item[(1)] Modified Bing-Borsuk conjecture \cite{br1}: If $X$ is an $n$-dimensional homogeneous $ANR$-space, then $X$ is a homology $n$-manifold.
\item[(2)] Homogeneity conjecture \cite{bfmw}: Every connected $ANR$ homology $(n\geq 5)$-manifold satisfying the disjoint disks property is homogeneous.  
\end{itemize}
Some partial answers to the Modified Bing-Borsuk conjecture are given in \cite{bre}, \cite{br} and \cite{br2}. For example Bredon \cite{bre} and Bryant \cite{br2} proved that if $X$ is a homogenous locally compact $ANR$ of dimension $n$ such that all groups $H_k(X,X\backslash\{x\})$, $k\leq n$, are finitely generated, then $X$ is a homology $n$-manifold.

In the present note we provide a necessary and sufficient condition locally compact homogeneous $ANR$-spaces or strongly locally homogeneous $ANR$-spaces  to be almost homology $n$-manifolds. 
We have the following results, where $\mathbb B^{n-1}$ denotes the $(n-1)$-dimensional ball.
\begin{thm}
Let $X$ be a locally compact homogeneous $ANR$-space. Then $X$ is an almost homology $n$-manifold if and only if $X$ satisfies the following condition $B(n-1)$: the set of all maps $f:\mathbb B^{n-1}\to X$ such that $f(\mathbb B^{n-1})$ has an empty interior is dense in the function space $C(\mathbb B^{n-1},X)$ equipped with the compact open topology.
\end{thm}
  
\begin{thm}
Let $X$ be a strongly locally homogeneous $LC^{n-1}$-space. Then $X$ is an almost homology $n$-manifold if and only $X$ has the property $B(n-1)$.
\end{thm}
Every nowhere locally compact space satisfies the condition $B(n)$ for every $n$. Moreover, if $X$ is an almost homology $n$-manifold then $X$ is of dimension $\geq n$, see Proposition \ref{dim} below. Therefore, Theorem 1.2 implies the following corollaries:
\begin{cor}
A strongly locally homogeneous $LC^{n-1}$-space is an almost homology $n$-manifold provided it is  nowhere locally compact.
\end{cor}
\begin{cor}
There is no  finite-dimensional strongly locally homogeneous $ANR$-space which is  nowhere locally compact.
\end{cor}
Recall that $X$ is an $ANR$-space \cite{bo} if for every embedding of $X$ as a closed subset of a metrizable space $Y$ there is a neighborhood $U$ of $X$ in $Y$ and a retraction $r:U\to X$, i.e. a continuous map $r$ with $r(x)=x$ for all $x\in X$. Every $ANR$-space $X$ is locally contractible, so it is $LC^k$ for all $k$, i.e. for every $x\in X$ and its neighborhood there exists another neighborhood $V$ of $x$ such that the inclusion $V\hookrightarrow U$ indices trivial homomorphisms between the homotopy groups $\pi_m(V)$ and $\pi_m(U)$ for all $m\leq k$.

The following condition $LS^{n-2}$ was considered in \cite{tv}: for every $x\in X$ and its neighborhood $U$ there is another neighborhood $V\subset U$ such that every map $f:\mathbb S^{n-2}\to V$ is approximated by maps $f'\in C(\mathbb S^{n-2},V)$ with $f'(\mathbb S^{n-2})$ not separating $U$ (here $\mathbb S^{n-2}$ is the $(n-2)$-dimensional sphere).
It was shown \cite[Theorem 4.5]{tv} that every point of a locally compact homogeneous $ANR$-space satisfying $LS^{n-2}$ is $Z_{n-1}$-set. According to Theorem 1.1, condition $LS^{n-2}$ implies $B(n-1)$ for every locally compact homogeneous $ANR$-space. 

\section{Prof of Theorem 1.1}
A closed set $A\subset X$ is a $Z_n$-set in $X$ if the function space $C(\mathbb B^n,X)$ contains a dense subset of maps $f$ with
$f(\mathbb B^n)\cap A=\varnothing$. This is equivalent to the following: if $\rho$ is any metric agreeing with the topology of $X$, then for every
$f\in C(\mathbb B^n,X)$ and every $\varepsilon >0$ there exists $g\in C(\mathbb B^n,X)$ such that $g(\mathbb B^n)\cap A=\varnothing$ and $\rho(f(z),g(z))<\varepsilon$ for all $z\in\mathbb B^n$. In such a case we say that $f$ and $g$ are $\varepsilon$-close.
It is easily seen that every $Z_n$-set in a space $X$ is $Z_m$ for every $m\leq  n$.
There is a homology analogue of $Z_n$-sets: a closed set $A\subset X$ is called a homological $Z_n$-set in $X$ \cite{bck} if $H_k(U,U\backslash A)=0$ for all $k\leq n$ and all open sets $U\subset X$. It follows from the excision axiom that a point $x\in X$ is a homological $Z_n$-set in $X$ if and only $H_k(X,X\backslash\{x\})=0$ for all $k\leq n$. 
We have the following relations between $Z_n$-sets and homological $Z_n$-sets:
\begin{pro}\cite{bck} Let $X$ be a space. Then
\begin{itemize}
\item[(i)] If $X$ is $LC^n$, then a set $A\subset X$ is $Z_n$ if and only $A$ is homotopical $Z_n$, i.e. for any admissible metric $\rho$ on $X$, $\varepsilon >0$ and a map $f\in C(\mathbb B^n,X)$ there is a map $f':\mathbb B^n\to X\backslash A$ and an $\varepsilon$-small homotopy between  $f$ and $f'$.
\item[(ii)] Each homotopical $Z_n$-set in $X$ is a $Z_n$-set in $X$.
\item[(iii)] Each homotopical $Z_n$-set in $X$ is a homological $Z_n$-set in $X$.
\item[(iv)] Each homological $Z_k$-set in $X$ is a $Z_k$-set in $X$ for any $k\in\{0,1\}$. 
\item[(v)] If $X$ is $LC^1$, then a homotopical $Z_2$-set in $X$ is a homotopical $Z_n$-set in $X$ if and only if $A$ is a homological $Z_{n}$-set in $X$.
\end{itemize}
\end{pro}
\begin{re}
Not every $Z_n$-set is a homological $Z_n$ even for $LC^{n-1}$-spaces. For example, every point of the $n$-dimensional universal Menger compactum $\mu^n$ is a $Z_n$-set \cite{be}, but not a homological $Z_n$-set. Indeed, otherwise $\mu^n$ would be an almost homology $(n+1)$-manifold, so by Proposition 3.1  $\dim\mu^n\geq n+1$. 
\end{re}
\begin{thm} Let $X$ be a locally compact homogeneous $ANR$-space. Then $X$ 
is an almost homology $n$-manifold if and only if $\{x\}$ is a $Z_{n-1}$-set in $X$ for all $x\in X$.
\end{thm}
\begin{proof}
If $X$ is a locally compact almost homology $n$-manifold, then every $\{x\}$ is a homological $Z_{n-1}$-set in $X$. Then by Proposition2.1(iv), every point is a $Z_{n-1}$-set in $X$ provided $n\in\{1,2\}$. If $n\geq 3$,  
according to \cite[Proposition 1.8 and Theorem 2.6]{kr}, every 
$\{x\}$ is a homotopical $Z_2$-set in $X$. Hence, by Proposition 2.1(v), all points are $Z_{n-1}$-set in $X$. The other implication follows from the fact that every $Z_{n-1}$-set in an $ANR$-space is a homological $Z_{n-1}$-set, see Proposition 2.1.
\end{proof}

{\em Proof of Theorem $1.1$.} Suppose $X$ is a locally compact homogeneous $ANR$-space which is an almost homology $n$-manifold. By Theorem 2.3, every point of $X$ is a $Z_{n-1}$-set in $X$. Choose a countable dense subset $\{x_k\}_{k\geq 1}$ of $X$. Then for each $k$ the set $G_k$ of all maps $f\in C(\mathbb B^{n-1},X)$ with $x_k\not\in f(\mathbb B^{n-1})$ is open in dense in $C(\mathbb B^{n-1},X)$. So, $G=\bigcap_{k\geq 1}G_k$ is also dense in $C(\mathbb B^{n-1},X)$. Observe that the interior of $f(\mathbb B^{n-1})$ is empty for every $f\in G$. 

To proof the other implication of Theorem 1.1, we need the following result of Krupski \cite{kr}. 
\begin{pro}\cite{kr}\label{effros}
Let $X$ be a locally compact homogeneous space and $\rho$ be a metric generating the topology of the one-point compactification $X^*$ of $X$.
Suppose $U$ is an open neighborhood of $x\in X$ with compact closure, $0<\varepsilon<\frac{1}{2}\rho(x,X\backslash U)$ and $K=X\backslash N_{\varepsilon}(X\backslash U)$, where $N_{\varepsilon}(X\backslash U)$ denotes the open $\varepsilon$-neighborhood of $X\backslash U$. Then there exists 
a $\delta>0$ such that if $y\in U$ and $\rho(x,y)<\delta$, we have a mapping $\lambda:U\to U$ which is $\varepsilon$-close to the identity on $U$, $\lambda|K$ is a homeomorphism and $\lambda(y)=x$.
\end{pro}

Let prove the sufficiency in Theorem 1.1.
Suppose $X$ is a locally compact homogeneous $ANR$-space with the property $B(n-1)$. If $n=1$, then $X$ doesn't contain any isolated point. Hence, every point of $X$ is a $Z_0$-set. If $n\geq 2$, we use the well-known fact 
(see for example \cite[Lemma 2.1]{vv1}) that a closed nowhere dense subset $A$ of a metrizable $LC^{k-1}$-space $Y$ is a $Z_k$-set in $Y$ iff $A$ is $LCC^{k-1}$ in $Y$. This means that for every $y\in A$ and its neighborhood $U$ in $Y$ there is another neighborhood $V$ of $y$ in $Y$ such that $V\subset U$ and every map $f:\mathbb S^m\to V\backslash A$ has a continuous extension $\widetilde f:\mathbb B^{m+1}\to U\backslash A$, where $m\leq k-1$. 
So, it suffices to show that every $\{x\}$ is $LCC^{k}$ in $X$ for all $k\leq n-2$. 
\begin{claim}
$X$ has the property $B(k)$ for all $k\leq n-1$
\end{claim}   
This is true for $k=n-1$, so let $k<n-1$ and $f\in C(\mathbb B^k,X)$, where $\mathbb B^k$ is assumed to be a subset of $\mathbb B^{n-1}$. Since the restriction map 
$\pi_{\mathbb B^k}:C(\mathbb B^{n-1},X)\to C(\mathbb B^k,X)$ is surjective, there is $\widetilde f\in C(\mathbb B^{n-1},X)$ with $\pi_{\mathbb B^k}(\widetilde f)=f$. Then $\widetilde f$ is approximated by maps $\widetilde g\in C(\mathbb B^{n-1},X)$ such that the interior of all $\widetilde g(\mathbb B^{n-1})$ is empty. Hence, the maps $g=\widetilde g|\mathbb B^k$ approximate $f$ and the interior of each $g(\mathbb B^k)$ is also empty.

\begin{claim}
Every point of $X$ is $LCC^k$, where $k\leq n-2$. 
\end{claim}
Let $X^*$ be the one-point compactification of $X$ and $\rho$ be a metric on $X^*$ generating its topology. Take a point $p\in X$ and a connected open neighborhood $U$ of $p$ in $X$ and let $V$ be another neighborhood of $p$ such that $\rm{cl} V$ is contractible in $U$.
 Choose any map $f:\mathbb S^k\to V\backslash\{p\}$. We are going to show that $f$ can be extended to a map
$\widetilde f:\mathbb B^{k+1}\to U\backslash\{p\}$. Since $U\backslash\{p\}$ has the homotopy extension property, it suffices to show that $f$ can be approximated by maps $f':\mathbb S^k\to V\backslash\{p\}$ such that each $f'$ has an extension $\widetilde f':\mathbb B^{k+1}\to U\backslash\{p\}$. 
To this end, we extend $f$ to a map $\widetilde f:\mathbb B^{k+1}\to U$. According to Claim 1, $\widetilde f$ is approximated by maps $\widetilde g:\mathbb B^{k+1}\to U$ such that each $\widetilde g(\mathbb B^{k+1})$ has an empty interior. Then the restriction maps $g=\widetilde g|\mathbb S^k$ approximate $f$ and  
$g(\mathbb S^k)\subset V\backslash\{p\}$ for each $\widetilde g$ close enough to $\widetilde f$. So, the proof is reduced to show that any such $g$ can be extended to a map from $\mathbb B^{k+1}$ into $U\backslash\{p\}$. To this end, let $W$ be the path-component of   
$U\backslash g(\mathbb S^k)$ containing the point $p$. Since $U\backslash g(\mathbb S^k)$ is locally path-connected, $W$ is open in $U$. Hence, 
there is a point $b\in W\backslash\widetilde g(\mathbb B^{k+1})$ (recall that the interior of $\widetilde g(\mathbb B^{k+1})$ is empty).
Hence, there is an arc $P\subset W$ joining the points $b$ and $p$. 

Next, following an idea of Krupski from the proof of \cite[Theorem 2.6]{kr}, consider the set
$$M=\{z\in P:\hbox{ }\mbox{there exists a map}\hbox{~} h_z\colon\mathbb B^{k+1}\to U\backslash\{z\}\hbox{~}\mbox{extending}\hbox{~}g\}.$$
Obviously, $M$ is open in $P$ and $b\in M$. We are going to show $M=P$. 
Because $P$ is connected, it suffices to show $M$ is closed in $P$.
So, let $x\in\rm{cl}~M$ and  $\varepsilon<\frac{1}{2}\rho(x,X\backslash U)$  be a positive number such that if a map $g':\mathbb S^k\to U\backslash\{x\}$ is $\varepsilon$-close to $g$, then $g'$ and $g$ are homotopic in $U\backslash\{x\}$. Let $\delta$ be the number from Proposition \ref{effros} corresponding to $\varepsilon$ and $x$. Take a point $y\in M$ with $\rho(x,y)<\delta$. Then there exist a map
$h_y:\mathbb B^{k+1}\to U\backslash\{y\}$ extending $g$ and a map $\lambda:U\to U$ such that $\lambda(y)=x$, $\lambda$ is a homeomorphism on the set $X\backslash N_{\varepsilon}(X\backslash U)$ and $\lambda$ is $\varepsilon$-close to the identity map on $U$. Let $h=\lambda\circ h_y$. Evidently, $h$ maps $\mathbb B^{k+1}$ into $U\backslash\{x\}$. Since $h|\mathbb S^k$ and $g$ are $\varepsilon$-close, they are homotopic in $U\backslash\{x\}$. Hence, $g$ can be extended to a map from $\mathbb B^{k+1}$ to $U\backslash\{x\}$. This means that $x\in M$, so $M$ is closed in $P$.
Therefore, $p\in M$, which provides a map $h_p:\mathbb B^{k+1}\to U\backslash\{p\}$ extending $g$.
Thus, $\{p\}$ is a $Z_{n-1}$-set in $X$.  $\Box$

\section{Proof of Theorem 1.2}
\begin{pro}\label{dim}
If $X$ is an  almost homology $n$-manifold, then $\dim X\geq n$.  
\end{pro}

\begin{proof}
This fact was established in \cite[Proposition 1.7]{kr} in case $X$ is locally compact. A similar proof works and in the general case. Indeed, choose a connected open set $U\subset X$. If $\dim X<n$ then there exists a closed subset $\Gamma\subset X$ such that $\dim\Gamma\leq n-2$ and $\Gamma$ separates $U$. So, $U\backslash\Gamma$ is disconnected while $U$ is path-connected. Because each point of $X$ is a homological $Z_{n-1}$-set and $\dim\Gamma\leq n-2$, according to \cite[Theorem 4.3]{bck}, $\Gamma$ is a homological $Z_1$-set in $X$. Therefore, we have the following exact sequence, where $\widetilde H$ denotes the reduced homology groups: 
$$...\to H_1(U,U\backslash\Gamma)\to\widetilde H_0(U\backslash\Gamma)\to\widetilde H_0(U)\to... $$   
Since both groups $H_1(U,U\backslash\Gamma)$ and $\widetilde H_0(U)$ are trivial, so is $\widetilde H_0(U\backslash\Gamma)$. This means that $U\backslash\Gamma$ is connected, a contradiction.
\end{proof}

If $\mathcal U$ is a cover of $X$, we say that two maps $f,g:Z\to X$ are $\mathcal U$-close if for every $z\in Z$ there is $U_z\in\mathcal U$ with
$f(z),g(z)\in U_z$. Similarly, the maps $f,g$ are $\mathcal U$-homotopic if there is a homotopy $F(z,t)$ between them such that the family $\{F(\{z\}\times\mathbb I):z\in Z\}$ refines $\mathcal U$. If in the above definition, $\rho$ is a metric on $X$ and $\mathcal U$ consists of all open balls $B(x,\varepsilon)=\{y\in X:\rho(x,y)<\varepsilon\}$, $x\in X$, any $\mathcal U$-close (resp., $\mathcal U$-homotopic) maps are called $\varepsilon$-close (resp., $\varepsilon$-homotopic).

The following fact follows from the properties of metrizable $LC^k$-spaces.
\begin{lem}
Let $(Y,\rho)$ be a metric $LC^{k}$-space and $f:\mathbb S^k\to Y$ be a map. Then there is a $\delta>0$ such that $f$ and $g$ are homotopic for every $\delta$-close map $g:\mathbb S^k\to Y$ to $f$. 
Moreover, by the Homotopy Extension Theorem, if a map $g:\mathbb S^k\to Y$ is homotopic to $f$ and $g$ is extendable to a map $\widetilde g:\mathbb B^{k+1}\to Y$, then $f$ is also extendable over $\mathbb B^{k+1}$.
\end{lem}
\begin{pro}
Let $X$ be a strongly locally homogeneous $LC^2$-space which is almost homology $n$-manifold with $n\geq 3$. Then every point of $X$ is a homotopical $Z_2$-set in $X$.  
\end{pro}
\begin{proof}
We fix a metric $\rho$ generating the topology of $X$. 
\begin{claim}
Let $W$ be a connected open set in $X$ and $p\in W$. Then the space $C(\mathbb S^1,W)$ contains a dense set of maps $f$ such that $f(\mathbb S^1)$ doesn't separate $W$.
\end{claim}
Choose a countable dense set $\{x_k:k\geq 1\}$ in $W$. Since $W$ is connected and locally connected complete metric space, it is path-connected. So, for every $k,m$ there is an arc $A(k,m)$ joining the points $x_k$ and $x_m$. Because the points of $W$ are homological $Z_{n-1}$-sets in $W$, by \cite[Theorem 4.3]{bck}, each $A(k,m)$ is a homological $Z_1$-set in $W$. Hence, according to Proposition 2.1(iv), all $A(k,m)$ are $Z_1$-sets. Therefore, the sets $G(k,m)=\{f\in C(\mathbb S^1,W):f(\mathbb S^1)\cap A(k,m)=\varnothing\}$ are open and dense in $C(\mathbb S^1,W)$ (the density of $G(k,m)$ follows from the fact that $\mathbb S^1$ is a union of two sets each homeomorphic to $\mathbb B^1$). 
So, the set $G=\bigcap_{k,m}G(k,m)\cap G_0$ is also dense in $C(\mathbb S^1,W)$. We claim that $f(\mathbb S^1)$ doesn't separate $W$ for all $f\in G$. Indeed, if $f(\mathbb S^1)$ separates $W$ for some $f\in G$, we can choose points $x_k$ and $x_m$  contained in different components of $W\backslash f(\mathbb S^1)$ (recall that the components of $W\backslash f(\mathbb S^1)$ are open because $W$ is locally connected). Consequently, the arc $A(k,m)$ should intersect $f(\mathbb S^1)$, a contradiction. 

Since $X$ is $LC^2$, the homotopical $Z_2$-sets in $X$ coincide with $Z_2$-sets in $X$. Moreover, $X$ does not contain isolated point because every $x\in X$ is a homological $Z_0$-set and hence a $Z_0$-set. 
Therefore, by \cite[Lemma 2.1]{vv1}, the following claim completes the proof. 
\begin{claim}
Every point of $X$ is $LCC^1$. 
\end{claim}
We take a point $p\in X$ and neighborhoods $U, V$ of $p$ such that $U$ is connected, $V\subset U$ and every map from $l:\mathbb S^1\to V$ is extendable to a map $\widetilde l:\mathbb B^2\to U$. We show that 
any map $f:\mathbb S^1\to V\backslash\{p\}$ can be extended to a map from $\mathbb B^2$ to $U\backslash\{p\}$. Since, by Claim 3, $f$ can be approximated by maps $f'\in C(\mathbb S^1,U\backslash\{p\})$ such that $U\backslash f'(\mathbb S^1)$ is connected, according to Lemma 3.2, we can assume that $f$ also has that property. Next, extend $f$ to a map $\widetilde f:\mathbb B^2\to U$ and choose an arc $P\subset U\backslash f(\mathbb S^1)$ joining $p$ and a point $b\in U\backslash\widetilde f(\mathbb B^2)$.  As in the proof of Claim 2, it suffices to show that the set 
$$M=\{z\in P:\hbox{ }\mbox{there exists a map}\hbox{~} h_z\colon\mathbb B^{2}\to U\backslash\{z\}\hbox{~}\mbox{extending}\hbox{~}f\}$$
coincides with $P$. Because $b\in M$ and obviously $M$ is open in $P$, the crucial step is to prove $M$ is closed in $P$. So, let $x\in\rm{cl}M$ and $\varepsilon=\min\{\rho(f(\mathbb S^1),X\backslash U),\rho(x,f(\mathbb S^1))\}$. Since $U\backslash\{x\}$ is $LC^2$, there exists a positive number $\delta<\varepsilon$ satisfying the hypotheses of Lemma 3.2 for the space $U\backslash\{x\}$. 
 Next, using the strong local homogeneity of $X$, we choose a neighborhood $O_x$ of diameter $<\delta$ and a homeomorphism $g$ on $X$ such that
$g$ is the identity on $X\backslash O_x$ and $g(y)=x$, where $y\in M\cap O_x$. Let $h=g\circ h_y$. Then $h$ is a map from $\mathbb B^2$ into $U\backslash\{x\}$ and, since $\rho(g(z),z)<\delta$ for all $z\in X$, $h|\mathbb S^1$ and $f$ are $\delta$-close. Consequently, they are homotopic in $U\backslash\{x\}$. Hence, by the Homotopy Extension Theorem, $f$ can be extended to a map from $\mathbb B^2$ to $U\backslash\{x\}$, which means that $x\in M$. Therefore, $M$  is a nonempty open and closed subset of $P$, so $M=P$.
\end{proof}
Proposition 3.3 and Proposition 2.1 imply the following (see the proof of Theorem 2.3):
\begin{cor}
Let $X$ be a strongly locally homogeneous $LC^{n-1}$-space. Then $X$ is an almost homology $n$-manifold if and only if $\{x\}$ is a $Z_{n-1}$-set in $X$ for all $x\in X$. 
\end{cor}

{\em Proof of Theorem $1.2$.} The proof is almost identical with the proof of Theorem 1.1. The only difference is the proof that every point of $X$ is $LCC^k$ for all $k\leq n-2$ assuming that $X$ is a strongly locally homogeneous $LC^{n-1}$-space satisfying the condition $B(n-1)$. Choose a point $p\in X$, its connected neighborhood $U$ in $X$ and another neighborhood $V$ of $p$ such that $V\subset U$ and every map $f':\mathbb S^k\to V$ is extendable to a map $\widetilde f':\mathbb B^{k+1}\to U$. 
Let $f:\mathbb S^k\to U\backslash\{p\}$.
Following the arguments from Claim 2, we extend $f$ to a map $\widetilde f:\mathbb B^{k+1}\to U$ and approximate $\widetilde f$ by maps $\tilde g:\mathbb B^{k+1}\to U$ such that 
the interior of $\tilde g(\mathbb B^{k+1})$ is empty. Then the restrictions $g=\tilde g|\mathbb S^k$ approximate $f$, and according to Lemma 3.2, it suffices to prove that any such a map $g$ can be extended to a map from $\mathbb B^{k+1}$ into $U\backslash\{p\}$. Next, take a path-component $W$ of $U\backslash g(\mathbb S^k)$ containing the point $p$ and choose a point $b\in W\backslash\widetilde f(\mathbb B^{k+1})$ and an arc $P\subset W$ joining $p$ and $b$.
According to Claim 2, that proof is reduced to show that the set
$M=\{z\in P:\hbox{ }\mbox{there exists a map}\hbox{~} h_z\colon\mathbb B^{k+1}\to U\backslash\{z\}\hbox{~}\mbox{extending}\hbox{~}g\}$
coincides with $P$. Finally, the proof of the last fact follows the arguments from Claim 4 with $\mathbb S^1$ and $\mathbb B^2$ replaced by $\mathbb S^k$ and $\mathbb B^{k+1}$, respectively.  $\Box$





\end{document}